\documentclass[11pt,reqno]{amsart}
\usepackage{a4wide}
\usepackage{amssymb}
\usepackage{amsthm}
\usepackage{amsmath}
\usepackage{amscd}
\usepackage{verbatim,url}
\usepackage{color,cancel}
\usepackage[mathscr]{eucal}
\textheight8.8in \textwidth6.55in \numberwithin{equation}{section}

\newtheorem{theorem}{Theorem}

\newtheorem{corollary}[theorem]{Corollary}

\theoremstyle{remark}

\numberwithin{theorem}{section} \numberwithin{equation}{section}
\numberwithin{figure}{section}

\renewcommand{\o}{\overline}

\newcommand{\field}[1]{\mathbb{#1}}
\newcommand{\C}{\field{C}}
\newcommand{\N}{\mathbb{N}}

\begin{document}
\title{Generalizations of Capparelli's identity}
\author{Jehanne Dousse}
\address{Institut f\"ur Mathematik, Universit\"at Z\"urich\\ Winterthurerstrasse 190, 8057 Z\"urich, Switzerland}
\email{jehanne.dousse@math.uzh.ch}

\author{Jeremy Lovejoy}
\address{CNRS, Universit\'e Denis Diderot - Paris 7, Case 7014, 75205 Paris Cedex 13, France}
\email{lovejoy@math.cnrs.fr}

\thanks{The first author is supported by the Forschungskredit of the University of Zurich, grant no. FK-16-098. The authors thank the University of Zurich and the French-Swiss collaboration project no. 2015-09 for funding research visits during which this research was conducted.}


\begin{abstract}
Using jagged overpartitions, we give three generalizations of a weighted word version of Capparelli's identity due to Andrews, Alladi, and Gordon and present several corollaries.   
\end{abstract}

\maketitle

\section{Introduction}

Recall that a partition $\lambda$ of an integer $n$ is a non-increasing sequence $(\lambda_1, \dots , \lambda_s)$ of positive integers whose sum is $n$.  Among the most celebrated results in the theory of partitions are the following, which are known as the Rogers-Ramanujan identities.
\begin{theorem}[The Rogers-Ramanujan identities]
Let $a=0$ or $1$. For positive $n$, the number of partitions of $n$ such that the difference between two consecutive parts is at least $2$ and the part $1$ appears at most $1-a$ times is equal to the number of partitions of $n$ into parts congruent to $\pm (1+a) \mod 5.$
\end{theorem}
In~\cite{Lepowsky}, Lepowsky and Wilson gave a connection between these identities and certain representations of the affine Lie algebra $sl_2(\C)^{\sim}.$ Inspired by their work, Capparelli~\cite{Capparelli1}, Meurman-Primc~\cite{Meurman}, Primc~\cite{Primc}, Siladi\'c~\cite{Siladic} and others considered other representations and obtained many interesting partition identities yet unknown to the combinatorics community. The most widely studied of these has probably been Capparelli's identity, which he conjectured in \cite{Capparelli1}.

\begin{theorem}[The Capparelli identity]
\label{th:capa}
Let $\mathcal{C}(n)$ denote the number of partitions of $n$ into parts $>1$ such that parts differ by at least $2$, and at least $4$ unless consecutive parts add up to a multiple of $3$.  Let $\mathcal{D}(n)$ denote the number of partitions of $n$ into distinct parts not congruent to $\pm 1 \pmod{6}$. Then for every positive integer $n$, $\mathcal{C}(n) = \mathcal{D}(n)$.   
\end{theorem}
Several proofs of Theorem \ref{th:capa} have been given. It was first proved combinatorially by Andrews in~\cite{Andrews1}, then by Capparelli in~\cite{Capparelli2} and by Tamba-Xie in~\cite{Xie} using Lie algebraic techniques, and then it was refined and generalized by Alladi, Andrews and Gordon in ~\cite{Al-An-Go1} using the method of ``weighted words.''  For other combinatorial studies related to the Capparelli identity, see \cite{Be-Un1,Br-Ma1,Fu-Ze1,Si1}.

Our starting point here is the work of Alladi, Andrews, and Gordon \cite{Al-An-Go1}.    They considered partitions into natural numbers in three colors, $a$, $b$, and $u$ (they wrote $c$ instead of $u$), with no part $1_b$ or $1_a$, 
satisfying the difference conditions in the matrix
\begin{equation} \label{Cappdiffmatrix}
C=\bordermatrix{\text{} & a & b & u \cr a & 2 & 0 & 2 \cr b & 2 & 2 & 3 \cr u & 1 & 0 & 1}.
\end{equation}
Here the entry $(x,y)$ in the matrix $C$ gives the minimal difference between successive parts of colors $x$ and $y$.   
Letting $C(n;i,j)$ denote the number of such partitions of $n$ with $i$ parts colored $a$ and $j$ parts colored $b$, they found the following infinite product generating function.
\begin{theorem}[Alladi-Andrews-Gordon \cite{Al-An-Go1}]
\label{th:capa-aag}
We have
\begin{equation} \label{eq:capa-aag}
\sum_{n,i,j \geq 0} C(n;i,j) a^ib^jq^n =  (-q)_{\infty}(-aq^2;q^2)_{\infty}(-bq^2;q^2)_{\infty}.
\end{equation}
\end{theorem}
Here we have used the usual $q$-hypergeometric notation for  $n \in \N \cup \{+ \infty \}$,
$$
(a)_n = (a;q)_n = \prod_{k=0}^{n-1} (1-aq^k).
$$
 Under the dilation $q \rightarrow q^3$ and the translations $a \rightarrow aq^{-2}$ and $b \rightarrow bq^{-4}$, the difference conditions in \eqref{Cappdiffmatrix} become
\begin{equation} 
\bordermatrix{\text{} & a & b & u \cr a & 6 & 2 & 4 \cr b & 4 & 6 & 5 \cr u & 5 & 4 & 3},
\end{equation}
which are equivalent to those defining the partitions counted by $\mathcal{C}(n)$ in Theorem \ref{th:capa}.   The product in \eqref{eq:capa-aag} becomes the generating function for the partitions counted by $\mathcal{D}(n)$, and with the two extra  parameters $a$ and $b$, one obtains a refinement of Capparelli's identity.

\begin{corollary}[Alladi-Andrews-Gordon \cite{Al-An-Go1}] \label{cor:capa-aag}
Let $\mathcal{C}(n;i,j)$ denote the number of partitions counted by $\mathcal{C}(n)$ and $\mathcal{D}(n)$ in Theorem \ref{th:capa}, with $i$ parts congruent to $1$ modulo $3$ and $j$ parts congruent to $2$ modulo $3$.   Then $\mathcal{C}(n;i,j) = \mathcal{D}(n;i,j)$.
\end{corollary}

In this paper we generalize the work of Alladi-Andrews-Gordon in three different ways.   The combinatorial setting is that of \emph{jagged overpartitions}.\begin{footnote}{The term ``jagged'' was coined in \cite{Fo-Ja-Ma1} to refer to partitions where successive parts may have negative differences.}\end{footnote}    We consider sequences of non-negative integers $(\lambda_1,\dots,\lambda_s)$ whose sum is $n$, where parts occur in the three colors $a$, $b$, and $u$, the final occurrence of a positive part may be overlined, with $\lambda_s \not \in \{0_a,0_b,\overline{1}_a,\overline{1}_b\}$, and such that the miminal difference between successive parts $\lambda_i$ and $\lambda_{i+1}$ is given by the matrix
 \begin{equation} \label{Cappdiffmatrix-1}
\overline{C} = \bordermatrix{\text{} & \o{a} & \o{b} & \o{u} & a & b & u \cr \o{a} & 2 & 0 & 2 & 2 & 0 & 2 \cr \o{b} & 2 & 2 & 3 & 2 & 2 & 3 \cr \o{u} & 1 & 0 & 1 & 1 & 0 &1 \cr a & 1 & -1 & 1 & 1 & -1 & 1 \cr b & 1 & 1 & 2 & 1 & 1 & 2 \cr u & 0 & -1 & 0 & 0 & -1 &0}.
\end{equation}
For example, the jagged overpartition $(4_a, \overline{5}_b, 2_u,2_u,\overline{2}_u,1_u,2_b,1_a)$ satisfies the conditions above.   Let $\overline{C}(n;k;i,j)$ denote the number of such jagged overpartitions of $n$ with $k$ non-overlined parts, $i$ parts colored $a$ and $j$ parts colored $b$.    Note that when all parts are overlined (corresponding to $k=0$), the jagged overpartitions may be identified with the partitions counted by $C(n;i,j)$. Our three generalizations of Theorem \ref{th:capa-aag}  are the following.

\begin{theorem}
\label{th:capa-chelou}
Let  $\overline{C}_1(n;k;i,j)$ denote the number of jagged overpartitions counted by $\overline{C}(n;k,i,j)$ with the added condition that if the smallest $u$-colored part $x_u$ is equal to the number of overlined parts which come after $x_u$ (with $x_u$ included), then the final part must be overlined. Then we have
\begin{equation} 
\label{eq:capa-chelou}
\sum_{n,k,i,j \geq 0} \overline{C}_1(n;k;i,j) a^ib^jd^kq^n = \frac{(-q)_{\infty}}{(dq)_{\infty}} \sum_{r,s \geq 0} \frac{q^{r^2+r+s^2+s - \binom{r+s+1}{2}} a^rb^s d^{r+s} (-q/d)_{r+s}}{(q^2;q^2)_r(q^2;q^2)_s}.
\end{equation} 
\end{theorem}

\begin{theorem}
\label{th:capa-asym1}
Let $\overline{C}_2(n;k;i,j)$ denote the number of jagged overpartitions of $n$ counted by $\overline{C}(n;k;i,j)$ with the added condition that  the final $\ell$ parts must be overlined, where $\ell$ is the number of parts of color $b$ or $u$. Then
\begin{equation} \label{eq:capa-asym1}
\sum_{n,k,i,j \geq 0} \overline{C}_2(n;k;i,j) a^ib^jd^kq^n = (-q)_{\infty}(-bq^2;q^2)_{\infty} \sum_{r \geq 0} \frac{q^{\binom{r+1}{2}}(ad)^r(-q/d)_r}{(q^2;q^2)_r}.
\end{equation}
\end{theorem}

\begin{theorem}
\label{th:capa-asym2}
Let $\overline{C}_3(n;k;i,j)$ denote the number of jagged overpartitions of $n$ counted by $\overline{C}(n;k;i,j)$ with the added condition that the final $t$ parts must be overlined, where $t$ is the number of parts of color $u$.  Then
\begin{equation} \label{eq:capa-asym2}
\sum_{n,k,i,j \geq 0} \overline{C}_3(n;k;i,j) a^ib^jd^kq^n = (-q)_{\infty} \sum_{r,s \geq 0} \frac{q^{r^2+r+s^2+s - \binom{r+s+1}{2}} a^rb^s d^{r+s} (-q/d)_{r+s}}{(q^2;q^2)_r(q^2;q^2)_s}.
\end{equation}
\end{theorem}

Notice that if $d=0$ in equations \eqref{eq:capa-chelou} -- \eqref{eq:capa-asym2} then the only terms contributing to the sums on the left-hand sides are those with $k=0$, and it is easy to see that each of the $\overline{C}_i(n;0;i,j)$ is equal to the $C(n;i,j)$ in Theorem \ref{th:capa-aag}.   On the other hand, when $d=0$ on the right-hand sides the sums can be evaluated using  the $q$-binomial identity \cite{Ga-Ra1},
\begin{equation} \label{qbin1}
\sum_{n \geq 0} \frac{z^nq^{\binom{n+1}{2}}}{(q)_n} = (-zq)_{\infty},
\end{equation}
and we recover the product on the right-hand side of \eqref{eq:capa-aag}. It turns out that the sums also become infinite products when $d=1$, giving the following three results.  

\begin{theorem} \label{th:capa-chelounod}
We have
\begin{equation} 
\label{eq:capa-over}
\sum_{n,k,i,j \geq 0} \overline{C}_1(n;k;i,j) a^ib^jq^n = \frac{(-q)_{\infty}(-aq)_{\infty}(-bq)_{\infty}}{(q)_{\infty}(abq;q^2)_{\infty}}.
\end{equation}
\end{theorem}

\begin{theorem} \label{th:capa-asym1nod}
We have
\begin{equation} 
\label{eq:capa-asym1nod}
\sum_{n,k,i,j \geq 0} \overline{C}_2(n;k;i,j) a^ib^jq^n = (-q)_{\infty}(-aq)_{\infty}(-bq^2;q^2)_{\infty}.
\end{equation}
\end{theorem}

\begin{theorem} \label{th:capa-asym2nod}
We have
\begin{equation} 
\label{eq:capa-asym2nod}
\sum_{n,k,i,j \geq 0} \overline{C}_3(n;k;i,j) a^ib^jq^n = \frac{(-q)_{\infty}(-aq)_{\infty}(-bq)_{\infty}}{(abq;q^2)_{\infty}}.
\end{equation}
\end{theorem}

Theorems \ref{th:capa-chelou} -- \ref{th:capa-asym2nod} contain a number of interesting corollaries.    For example, if we let $q \to q^3, a \to aq^{-2}$ and $b \to bq^{-4}$ in Theorem \ref{th:capa-asym1} we obtain the following generalization of Capparelli's identity, where we use the notation
$$\chi(\o{\lambda_i}) = \begin{cases}
1 \text{ if $\lambda_i$ is overlined}, \\
0 \text{ otherwise.} \end{cases}$$

\begin{corollary} \label{cor1}
Let $\overline{\mathcal{C}}(n;k;i,j)$ denote the number of jagged overpartitions $(\lambda_1,\dots,\lambda_s)$ of $n$ into positive parts, $k$ of which are non-overlined, $i$ (resp. $j$) of which are congruent to $1$ (resp. $2$) modulo $3$,  with $\lambda_s \neq \overline{1}$, such that the difference between $\lambda_i$ and $\lambda_{i+1}$ is 
$$ \lambda_i -\lambda_{i+1} \geq \begin{cases}
-1+3 \chi(\o{\lambda_i}) \text{ if $\lambda_i + \lambda_{i+1} \equiv 0 \mod 3$}, \\
1 + 3 \chi(\o{\lambda_i}) \text{ otherwise},
\end{cases}$$
and the final $t$ parts are overlined, where $t$ is the number of parts congruent to $0$ or $2$ modulo $3$.    Let $\overline{\mathcal{D}}(n;k;i,j)$ denote the number of partitions of $n$ into distinct parts not congruent to $5 $ modulo $6$, $i$ (resp. $j$) of which are congruent to $1$ (resp. $2$) modulo $3$, with $k$ the number of times $\lambda^{(1)}_i - \lambda^{(1)}_{i+1} \equiv 3 \mod{6}$,  where $(\lambda^{(1)}_1,\dots,\lambda^{(1)}_s)$ is the subpartition of parts congruent to $1$ modulo $3$ (and we take $\lambda^{(1)}_{s+1} = -2$).     Then $\overline{\mathcal{C}}(n;k;i,j) = \overline{\mathcal{D}}(n;k;i,j)$.
\end{corollary}

Similarly, letting $q \to q^4, a \to aq, b \to bq^{-2}$ in Theorem \ref{th:capa-chelounod}, we obtain the following  identity.

\begin{corollary} \label{cor2}
Let $\overline{\mathcal{C}'}(n;i,j)$ denote the number of jagged overpartitions $(\lambda_1,\dots,\lambda_s)$ of $n$ into positive parts not congruent to $3$ modulo $4$, where $i$ (resp. $j$) is the number of parts congruent to $1$ (resp. $2$) modulo $4$, such that $\lambda_s \neq \overline{1}, 1, \o{2}, \o{5}$ and the difference between $\lambda_i$ and $\lambda_{i+1}$ is 
$$ \lambda_i -\lambda_{i+1} \geq \begin{cases}
-2+4 \chi(\o{\lambda_i}) \text{ if $\lambda_i \equiv 0 \mod 4$ or $\lambda_i + \lambda_{i+1} \equiv 3 \mod 4$}, \\
4 + 4 \chi(\o{\lambda_i}) \text{ otherwise},
\end{cases}$$
with the added condition that if the smallest part $x$ congruent to $0$ modulo $4$ is equal to $4$ times the number of overlined parts which come after $x$ (with $x$ included), then the final part must be overlined.
Let $\overline{\mathcal{D}'}(n;i,j)$ denote the number of overpartitions of $n$ with overlined parts not congruent to $3$ modulo $4$, no part $\o{1}$, and non-overlined parts congruent to $0,3,4$ modulo $8$, where $i$ is the number of parts congruent to $1 \mod 4$ or $3 \mod 8$ and $j$ is the number of parts congruent to $2 \mod 4$ or $3 \mod 8$. Then $\overline{\mathcal{C}'}(n;i,j) = \overline{\mathcal{D}'}(n;i,j)$.
\end{corollary}

Finally, by doing the same dilation in Theorem \ref{th:capa-asym2nod}, we obtain the following.

\begin{corollary} \label{cor3}
Let $\overline{\mathcal{C}''}(n;i,j)$ denote the number of jagged overpartitions $(\lambda_1,\dots,\lambda_s)$ of $n$ into positive parts not congruent to $3$ modulo $4$, where $i$ (resp. $j$) is the number of parts congruent to $1$ (resp. $2$) modulo $4$, such that $\lambda_s \neq \overline{1}, 1, \o{2}, \o{5}$ and the difference between $\lambda_i$ and $\lambda_{i+1}$ is 
$$ \lambda_i -\lambda_{i+1} \geq \begin{cases}
-2+4 \chi(\o{\lambda_i}) \text{ if $\lambda_i \equiv 0 \mod 4$ or $\lambda_i + \lambda_{i+1} \equiv 3 \mod 4$}, \\
4 + 4 \chi(\o{\lambda_i}) \text{ otherwise},
\end{cases}$$
and the final $t$ parts must be overlined, where $t$ is the number of parts congruent to $0 \mod 4$.
Let $\overline{\mathcal{D}''}(n;i,j)$ denote the number of overpartitions of $n$ with overlined parts not congruent to $3$ modulo $4$, no part $\o{1}$, and non-overlined parts congruent to $3$ modulo $8$, where $i$ is the number of parts congruent to $1 \mod 4$ or $3 \mod 8$ and $j$ is the number of parts congruent to $2 \mod 4$ or $3 \mod 8$.     Then $\overline{\mathcal{C}''}(n;i,j) = \overline{\mathcal{D}''}(n;i,j)$.
\end{corollary}

This paper is organized as follows.    In the next section we give a detailed proof of Theorem \ref{th:capa-aag} which is different from those of Alladi, Andrews and Gordon in \cite{Al-An-Go1}.   Ours uses some combinatorial reasoning with a staircase partition and $q$-series identities, while theirs also relied on properties of the $q$-binomial coefficients and/or $q$-difference equations.  Variations of our proof lead to Theorems \ref{th:capa-chelou} -- \ref{th:capa-asym2} and Theorems \ref{th:capa-chelounod} -- \ref{th:capa-asym2nod}.   These arguments are presented in Sections 3 and 4.   In Section 5 we discuss the combinatorial meaning of the $q$-series in Theorems \ref{th:capa-chelou} -- \ref{th:capa-asym2} and illustrate Corollaries \ref{cor1} -- \ref{cor3} with examples.   We close in Section 6 with a discussion of related work on a partition identity of Schur and some suggestions for future research.

\section{Proof of Theorem \ref{th:capa-aag}}
\label{sec:capa}
We begin with a proof of Theorem \ref{th:capa-aag} via a constant term identity.   This was alluded to at the end of \cite{Lo.5}, though no details were provided there.   We will show that   
\begin{equation} \label{constantterm}
[z^0] (-q/z)_{\infty}(-z)_{\infty} (q)_{\infty} \frac{(-azq)_{\infty}(-bzq)_{\infty}}{(z)_{\infty}(abz^2q;q^2)_{\infty}} = (-q)_{\infty}(-aq^2;q^2)_{\infty}(-bq^2;q^2)_{\infty}
\end{equation}
and then argue that the left-hand side is the generating function for $C(n;i,j)$.     To prove \eqref{constantterm}, we require the identities \cite{Ga-Ra1}
\begin{equation} \label{qnminusk}
(q)_{n-k} = \frac{(q)_n}{(q^{-n})_k}(-1)^kq^{\binom{k}{2} -nk},
\end{equation}
\begin{equation} \label{qChu}
\sum_{k=0}^n \frac{(a)_k(q^{-n})_kq^k}{(q)_k(c)_k} = \frac{(c/a)_na^n}{(c)_n},
\end{equation}
\begin{equation} \label{qbin2}
\sum_{n \geq 0} \frac{z^n(-x)_n}{(q)_n} = \frac{(-xz)_{\infty}}{(z)_{\infty}},
\end{equation}
and
\begin{equation} \label{JTP}
\sum_{n \in \mathbb{Z}} z^{-n}q^{\binom{n+1}{2}} = (-q/z)_{\infty} (-z)_{\infty} (q)_{\infty} .
\end{equation}
Then, expanding the products $(-azq)_{\infty}$ and $(-bzq)_{\infty}$ using \eqref{qbin1}, the products $1/(z)_{\infty}$ and $1/(abz^2q;q^2)_{\infty}$ using \eqref{qbin2}, and the product $ (-q/z)_{\infty} (-z)_{\infty} (q)_{\infty}$ using \eqref{JTP}, we proceed as follows:
\begin{align}
[z^0]& (-q/z)_{\infty}(-z)_{\infty} (q)_{\infty} \frac{(-azq)_{\infty}(-bzq)_{\infty}}{(z)_{\infty}(abz^2q;q^2)_{\infty}} \nonumber \\
&= [z^0] \sum_{n \in \mathbb{Z}} z^{-n}q^{\binom{n+1}{2}} \sum_{r,s,t,v \geq 0} \frac{z^{r+s+t+2v}q^{\binom{r+1}{2} + \binom{s+1}{2} + v}a^{r+v}b^{s+v}}{(q)_r(q)_s(q)_t(q^2;q^2)_v} \nonumber \\
&= \sum_{r,s,t,v \geq 0} \frac{q^{\binom{r+s+ t + 2v+1}{2} + \binom{r+1}{2} + \binom{s+1}{2} + v}a^{r+v}b^{s+u}}{(q)_r(q)_s(q)_t(q^2;q^2)_v} \label{keyplace} \\
&= \sum_{r,s,t \geq 0} \sum_{v =0}^{\min(r,s)} \frac{q^{\binom{r+s+ t +1}{2} + \binom{r - v+1}{2} + \binom{s -v +1}{2} + v}a^{r}b^{s}}{(q)_{r-v}(q)_{s-v}(q)_t(q^2;q^2)_v} \quad \quad \text{$(r,s) = (r-v, s-v)$} \nonumber \\
&=\sum_{r,s,t \geq 0} \sum_{v =0}^{\min(r,s)} \frac{q^{\binom{r+s+ t +1}{2} + \binom{r +1}{2} + \binom{s +1}{2} + v}a^{r}b^{s}(q^{-r})_v(q^{-s})_v}{(q)_{r}(q)_{s}(q)_t(q^2;q^2)_v} \quad \quad \text{(by \eqref{qnminusk})} \nonumber \\
&= \sum_{r,s,t \geq 0} \frac{q^{\binom{r+s+ t +1}{2} + \binom{r +1}{2} + \binom{s +1}{2} - rs}a^{r}b^{s}(-q)_{r+s}}{(q^2;q^2)_{r}(q^2;q^2)_{s}(q)_t} \quad \quad \text{(by \eqref{qChu})} \nonumber \\
&= \sum_{r,s,t \geq 0} \frac{q^{\binom{t+1}{2} + (r+s)t + r^2+r + s^2+s}a^{r}b^{s}(-q)_{r+s}}{(q^2;q^2)_{r}(q^2;q^2)_{s}(q)_t} \nonumber \\
&= (-q)_{\infty} \sum_{r,s \geq 0} \frac{q^{r^2+r + s^2+s}a^{r}b^{s}}{(q^2;q^2)_{r}(q^2;q^2)_{s}} \quad \quad \text{(by \eqref{qbin1})} \nonumber\\
&= (-q)_{\infty}(-aq^2;q^2)_{\infty}(-bq^2;q^2)_{\infty} \quad \quad \text{(by \eqref{qbin1} with $q=q^2$)}. \nonumber
\end{align}

Having confirmed \eqref{constantterm} we now turn to the combinatorial interpretation of 
\begin{equation}  \label{contributions}
[z^0] (-q/z)_{\infty}(-z)_{\infty} (q)_{\infty} \frac{(-azq)_{\infty}(-bzq)_{\infty}}{(z)_{\infty}(abz^2q;q^2)_{\infty}} = \sum_{r,s,t,v \geq 0} \frac{q^{\binom{r+s+ t + 2v+1}{2} + \binom{r+1}{2} + \binom{s+1}{2} + v}a^{r+v}b^{s+v}}{(q)_r(q)_s(q)_t(q^2;q^2)_v}.
\end{equation}
On the right-hand side, the term $a^rq^{\binom{r+1}{2}}/(q)_r$ contributes a partition $\lambda_a$ into $r$ distinct positive parts labelled $a$, the term $b^sq^{\binom{s+1}{2}}/(q)_s$ contributes a partition $\lambda_b$ into $s$ distinct positive parts labelled $b$, the term $(ab)^vq^v/(q^2;q^2)_v$ contributes a partition $\lambda_{ab}$ into $v$ positive odd parts labelled $ab$, and the term $1/(q)_t$ contributes a partition $\lambda_u$ into non-negative unlabelled parts, to which we shall assign the subscript $u$.   We organize these into one colored jagged partition $\lambda$ containing ``levels'' indexed by the natural numbers $k$ as follows:   First, odd parts $(2k-1)_{ab}$ from $\lambda_{ab}$ are split into pairs $(k-1)_a k_b$ and placed on the right.   Next, any extra $k_b$ from $\lambda_b$ is placed immediately to the left of these pairs.    Then we put all parts of the form $(k-1)_u$, and finally $k_a$ if it appears.    Thus for $k \geq 1$ the general level $k$ resembles
\begin{equation} \label{levelorder}
k_a (k-1)_u \cdots (k-1)_u k_b (k-1)_a k_b \cdots  (k-1)_a k_b .
\end{equation}
The levels are then put together in order from left to right with the smallest on the right.    For example, if $\lambda_a = (4,2,1)$, $\lambda_b = (4,3,1)$, $\lambda_{ab} = (7,7,7,5,1,1,1)$, and $\lambda_u = (1,1,0,0,0)$, we obtain
\begin{equation*}
\lambda = (\underbrace{4_a,4_b,3_a,4_b,3_a,4_b,3_a,4_b}_{\text{level $4$}},\underbrace{3_b,2_a,3_b}_{\text{level $3$}},\underbrace{2_a,1_u,1_u}_{\text{level $2$}},\underbrace{1_a,0_u,0_u,0_u,1_b,0_a,1_b,0_a,1_b,0_a,1_b}_{\text{level $1$}})
\end{equation*}

At this point we observe that the jagged partitions thus constructed are precisely those defined by $\overline{C}(n;k;i,j)$ where all parts are non-overlined, $i = r+v$ is the number of $a$-parts, and $j = s+v$ is the number of $b$-parts.   Indeed, one may check that the minimum difference between consecutive parts is as in the matrix $\overline{C}$ in \eqref{Cappdiffmatrix-1}, while the condition $\lambda_s \not \in \{0_a,0_b\}$ evidently holds.    The only term we have not accounted for in \eqref{contributions} is $q^{\binom{r+s+t+2v+1}{2}}$, but this simply corresponds to adding a staircase to the jagged partition: $1$ to the smallest part, $2$ to the next smallest part, and so on.    The result is a colored partition counted by $C(n;i,j)$, giving
\begin{equation*}
\sum_{r,s,t,v \geq 0} \frac{q^{\binom{r+s+ t + 2v+1}{2} + \binom{r+1}{2} + \binom{s+1}{2} + v}a^{r+v}b^{s+v}}{(q)_r(q)_s(q)_t(q^2;q^2)_v} = \sum_{n,i,j \geq 0} C(n;i,j) a^ib^jq^n,
\end{equation*}
and thus establishing Theorem \ref{th:capa-aag}.   \qed

Before continuing we briefly discuss the possibility of modifying the order \eqref{levelorder} of the parts within a level before adding on the staircase.    The ordering we chose is the one required to obtain Capparelli's identity, but we could use others.  These would give slightly different combinatorial interpretations of \eqref{contributions}.     For example, suppose we put the parts $(k-1)_u$ all the way to the left of the level $k$ instead of between $k_a$ and $k_b$.    Then after adding the staircase onto the jagged overpartition the matrix of difference conditions in \eqref{Cappdiffmatrix} is replaced by
\begin{equation} \label{Cappdiffmatrix*}
C^*=\bordermatrix{\text{} & a & b & u \cr a & 2 & 0 & 3 \cr b & 2 & 2 & 3 \cr u & 0 & 0 & 1}.
\end{equation}
We have also introduced an extra condition due to the placement of $(k-1)_u$ ahead of $k_a$.   Namely, we cannot have a triple of parts with colors $(u,a,b)$ satisfying the the minimal differences $(0,0)$.   

Now, using the usual dilations $q=q^3$, $a=aq^{-2}$ and $b=bq^{-4}$, the matrix \eqref{Cappdiffmatrix*} becomes  
\begin{equation} 
\bordermatrix{\text{} & a & b & u \cr a & 6 & 2 & 7 \cr b & 4 & 6 & 5 \cr u & 2 & 4 & 3},
\end{equation}
and we have the following companion to Capparelli's identity.   
\begin{theorem}
Let $\mathcal{C}^*(n)$ denote the number of partitions of $n$ into parts greater than $1$ such that parts differ by at least $2$, and at least $5$ if neither the larger part nor the sum of the two parts is divisble by $3$, where no three successive parts may differ by $2$.  Let $\mathcal{C}^*(n;i,j)$ denote the number of partitions counted by $\mathcal{C}^*(n)$ having $i$ parts congruent to $1$ modulo $3$ and $j$ parts congruent to $2$ modulo $3$.   Then $\mathcal{C}^*(n;i,j) =\mathcal{D}(n;i,j)$.
\end{theorem}
  
Such companions were hinted at by Alladi, Andrews, and Gordon \cite{Al-An-Go1}, but apprently never given in detail.    We leave it to the interested reader to explore other modifications of the order \eqref{levelorder}.

\section{A generalized staircase and the proofs of Theorems \ref{th:capa-chelou} and \ref{th:capa-chelounod}}
\label{sec:chelou}
In this section we prove Theorems \ref{th:capa-chelou} and \ref{th:capa-chelounod} by modifying the arguments of the previous section, replacing the staircase by a generalized staircase.    By a generalized staircase we mean a partition $\mu$ into distinct parts of size at most $n$.   Such partitions are generated by $(-q/d)_nd^n$, where the exponent of $d$ counts the number of integers between $1$ and $n$ which do not occur as parts.   When $d=0$ we recover the ordinary staircase $n,\dots,2,1$.  

The generalized staircase is applied to a partition $\lambda$ with $n$ parts as follows:   for each part of size $k$ of $\mu$, we add $1$ to each of the $k$ largest parts of $\lambda$ and then overline the $k$th largest part.    The exponent of $d$ then counts the number of non-overlined parts in the resulting overpartition.    When $d=0$ this amounts to just adding the usual staircase to $\lambda$; $1$ to the smallest part, $2$ to the next smallest part, and so on.

Ideally we would like to remove the staircase corresponding to $q^{\binom{r+s+t+2u+1}{2}}$ in \eqref{contributions} and replace it by a generalized staircase $(-q/d)_{r+s+t+2u}d^{r+s+t+2u}$.    Unfortunately the resulting jagged overpartitions do not seem to have an interesting generating function.   Indeed, since $0_u$ may occur with repetition, an extra (and non-obvious) condition is necessary.

\begin{proof}[Proofs of Theorems \ref{th:capa-chelou} and \ref{th:capa-chelounod}]
We make the replacement
\begin{equation} \label{specialterm}
q^{\binom{r+s+t+2v+1}{2}} \to q^td^{r+s+t+2v}(-q/d)_{r+s+t+2v} (1+q^{r+s+2v+1}) 
\end{equation}
on the right-hand side of \eqref{contributions} to obtain
\begin{equation} \label{newgf}
\sum_{r,s,t,v \geq 0} \frac{q^td^{r+s+t+2v}(-q/d)_{r+s+t+2v} (1+q^{r+s+2v+1}) q^{\binom{r+1}{2} + \binom{s+1}{2} + v}a^{r+v}b^{s+v}}{(q)_r(q)_s(q)_t(q^2;q^2)_v}.
\end{equation}
As noted in the previous section, the term
\begin{equation}
\frac{q^{\binom{r+1}{2} + \binom{s+1}{2} + u}a^{r+v}b^{s+v}}{(q)_r(q)_s(q)_t(q^2;q^2)_v}
\end{equation}
contributes a jagged overpartition $\lambda$ into $r+s+t+2u$ parts counted by $\overline{C}(n;k;i,j)$ where all parts are non-overlined, $i=r+v$ is the number of $a$-parts and $j=s+v$ is the number of $b$-parts.   To account for the remaining term in \eqref{newgf}, we write
\begin{equation}
\begin{aligned}
q^td^{r+s+t+2v}(-q/d)_{r+s+t+2v} (1+q^{r+s+2v+1}) 
&= q^td^{r+s+t+2v}(-q/d)_{r+s+t+2v}  \\ &\qquad + q^{r+s+t+2v+1}d^{r+s+t+2v}(-q/d)_{r+s+t+2v}
\end{aligned}
\end{equation}
and consider two disjoint cases.
In the first case, the term $q^t$ means that we add one to each part colored $u$ in $\lambda$ and rearrange the $u$-colored parts according to \eqref{levelorder}.   Thus we have no $0_u$, and the term $d^{r+s+t+2u}(-q/d)_{r+s+t+2u}$ corresponds to adding a generalized staircase.     For the term $q^{r+s+t+2u+1}d^{r+s+t+2u}(-q/d)_{r+s+t+2u}$, we insert a part $0_u$ in $\lambda$ and then we apply the generalized staircase to the resulting partition almost as usual, except that the term $q^{r+s+t+2u+1}$ corresponds to the smallest part necessarily being overlined.    Note that the presence of a part $0_u$ before adding on the generalized staircase can be detected after the adding of the generalized staircase by the condition that the smallest uncolored part $x_u$ is equal to the number of overlined parts coming after $x_u$ (with $x_u$ included).   This is because each time the part which was originally $0_u$ is augmented by $1$, either it or a part after it becomes overlined.    Thus, taking the two cases together, we have 
\begin{equation} \label{Cbar14vargf}
\begin{aligned}
\sum_{n,k,i,j \geq 0} &\overline{C}_1(n;k;i,j) a^ib^jd^kq^n \\
&= \sum_{r,s,t,v \geq 0} \frac{q^td^{r+s+t+2v}(-q/d)_{r+s+t+2v} (1+q^{r+s+2v+1}) q^{\binom{r+1}{2} + \binom{s+1}{2} + v}a^{r+v}b^{s+v}}{(q)_r(q)_s(q)_t(q^2;q^2)_v}.
\end{aligned}
\end{equation}      

To obtain the $q$-series on the right-hand side of \eqref{eq:capa-chelou}, we begin with the right-hand side above and follow the same steps as in the previous section after \eqref{keyplace}:
\begin{align*}
\sum_{r,s,t,v \geq 0} & \frac{q^td^{r+s+t+2v}(-q/d)_{r+s+t+2v} (1+q^{r+s+2v+1}) q^{\binom{r+1}{2} + \binom{s+1}{2} + v}a^{r+v}b^{s+v}}{(q)_r(q)_s(q)_t(q^2;q^2)_v} \\
&= \sum_{r,s,t \geq 0} \sum_{v=0}^{\min(r,s)} \frac{q^td^{r+s+t}(-q/d)_{r+s+t} (1+q^{r+s+1}) q^{\binom{r+1}{2} + \binom{s+1}{2} + v}a^{r}b^{s}(q^{-r})_v(q^{-s})_v}{(q)_r(q)_s(q)_t(q^2;q^2)_v}  \\
&= \sum_{r,s,t \geq 0} \frac{q^td^{r+s+t}(-q/d)_{r+s}(-q^{r+s+1}/d)_t (-q)_{r+s+1} q^{\binom{r+1}{2} + \binom{s+1}{2} -rs}a^{r}b^{s}}{(q^2;q^2)_r(q^2;q^2)_s(q)_t} \\
&= \frac{(-q)_{\infty}}{(dq)_{\infty}} \sum_{r,s \geq 0} \frac{q^{r^2+r+s^2+s - \binom{r+s+1}{2}} a^rb^s d^{r+s} (-q/d)_{r+s}}{(q^2;q^2)_r(q^2;q^2)_s} \quad \quad \text{(by \eqref{qbin2})} .
\end{align*}

To obtain the product in \eqref{eq:capa-over} we set $d=1$ in \eqref{Cbar14vargf}, carry out the sum over $t$ directly and then the sums over $r,s$, and $v$:
\begin{align*}
\sum_{r,s,t,v \geq 0} & \frac{q^t(-q)_{r+s+t+2v} (1+q^{r+s+2v+1}) q^{\binom{r+1}{2} + \binom{s+1}{2} + v}a^{r+v}b^{s+v}}{(q)_r(q)_s(q)_t(q^2;q^2)_v} \\
&= \sum_{r,s,t,v \geq 0} \frac{q^t(-q)_{r+s+2v} (-q^{r+s+2v+1})_t(1+q^{r+s+2v+1}) q^{\binom{r+1}{2} + \binom{s+1}{2} + v}a^{r+v}b^{s+v}}{(q)_r(q)_s(q)_t(q^2;q^2)_v} \\
&= \frac{(-q)_{\infty}}{(q)_{\infty}} \sum_{r,s,v \geq 0}  \frac{q^{\binom{r+1}{2} + \binom{s+1}{2} + v}a^{r+v}b^{s+v}}{(q)_r(q)_s(q^2;q^2)_v} \\
&= \frac{(-q)_{\infty}(-aq)_{\infty}(-bq)_{\infty}}{(q)_{\infty}(abq;q^2)_{\infty}}.
\end{align*}

\end{proof}

\section{Partial staircases and the proofs of Theorems \ref{th:capa-asym1} -- \ref{th:capa-asym2} and \ref{th:capa-asym1nod} -- \ref{th:capa-asym2nod}}
\label{sec:asym}  
In this section we modify the proof in Section 2 in a slightly different way from the previous section, by using partial staircases in place of a staircase or generalized staircase.    By a partial staircase we mean a generalized staircase which is a classical staircase at the top for a predetermined number of parts.      We consider two types of partial staircases, leading to Theorems \ref{th:capa-asym1} and \ref{th:capa-asym2}.   

\begin{proof}[Proof of Theorems \ref{th:capa-asym1} and \ref{th:capa-asym1nod}]
First, we make the replacement 
\begin{equation}
q^{\binom{r+s+t+2v+1}{2}} \to q^{\binom{r+s+t+2v+1}{2} - \binom{r+v+1}{2}} d^{r+v}(-q/d)_{r+v}
\end{equation}
on the right-hand side of \eqref{contributions}.     This means we have a partial staircase which is a staircase for the first $s+t+v$ parts and then a generalized staircase for the final $r+v$ parts.   Note that $s+t+v$ is the number of $b$-parts plus the number of $u$-parts.    Combinatorially, we will obtain jagged overpartitions corresponding to $\overline{C}(n;k;i,j)$ with the condition that the smallest $\ell$ parts are overlined, where $\ell = s+t+v$.    In other words, we have jagged overpartitions counted by $\overline{C}_2(n;k;i,j)$.     Using the usual $q$-series arguments we then obtain
\begin{align*}
\sum_{n,k,i,j \geq 0} &\overline{C}_2(n;k;i,j)a^ib^jd^kq^n  \\
&= \sum_{r,s,t,v \geq 0} \frac{q^{\binom{r+s+t+2v+1}{2} - \binom{r+v+1}{2}} d^{r+v}(-q/d)_{r+v}q^{\binom{r+1}{2} + \binom{s+1}{2} + v}a^{r+v}b^{s+v}}{(q)_r(q)_s(q)_t(q^2;q^2)_v} \\
&= \sum_{r,s,t \geq 0} \sum_{v=0}^{\min(r,s)} \frac{q^{rs+rt+st+\binom{s+1}{2} + \binom{t+1}{2} + \binom{r-v+1}{2} + \binom{s-v+1}{2} + v}a^rb^sd^r(-q/d)_r}{(q)_{r-v}(q)_{s-v}(q)_t(q^2;q^2)_v} \\
&= \sum_{r,s,t \geq 0} \sum_{v=0}^{\min(r,s)} \frac{q^{rs+rt+st+s^2+s + \binom{t+1}{2} + \binom{r+1}{2} + v}a^rb^sd^r(-q/d)_r(q^{-r})_v(q^{-s})_v}{(q)_{r}(q)_{s}(q)_t(q^2;q^2)_v} \\
&= \sum_{r,s,t \geq 0} \frac{q^{rt+st+s^2+s+ \binom{t+1}{2} + \binom{r+1}{2}}a^rb^sd^r(-q/d)_r(-q)_{r+s}}{(q^2;q^2)_{r}(q^2;q^2)_{s}(q)_t} \\
&= (-q)_{\infty} \sum_{s \geq 0} \frac{b^sq^{s^2+s}}{(q^2;q^2)_s}\sum_{r \geq 0} \frac{a^rd^r(-q/d)_rq^{\binom{r+1}{2}}}{(q^2;q^2)_r},
\end{align*}
and evaluating the sum over $s$ then completes the proof of Theorem \ref{th:capa-asym1}.    Theorem \ref{th:capa-asym1nod} then follows immediately by evaluating the sum over $r$ when $d=1$.
\end{proof}

\begin{proof}[Proof of Theorems \ref{th:capa-asym2} and \ref{th:capa-asym2nod}]
One could also make the replacement
\begin{equation}
q^{\binom{r+s+t+2v+1}{2}} \to q^{\binom{r+s+t+2v+1}{2} - \binom{r+s+2v+1}{2}} d^{r+s+2v}(-q/d)_{r+s+2v}
\end{equation}
on the right-hand side of \eqref{contributions}.     This corresponds to a partial staircase which is a staircase for the first $t$ parts and then a generalized staircase for the final $r+s+2v$ parts.   Note that $t$ is the number of $u$-parts and so we will obtain jagged overpartitions corresponding to $\overline{C}(n;k;i,j)$ with the condition that the smallest $t$ parts are overlined.    In other words, we have jagged overpartitions counted by $\overline{C}_3(n;k;i,j)$.     Using the usual $q$-series arguments we then obtain
\begin{align*}
\sum_{n,k,i,j \geq 0} &\overline{C}_3(n;k;i,j)a^ib^jd^kq^n  \\
&= \sum_{r,s,t,v \geq 0} \frac{q^{\binom{r+s+t+2v+1}{2} - \binom{r+s+2v+1}{2}} d^{r+s+2v}(-q/d)_{r+s+2v}q^{\binom{r+1}{2} + \binom{s+1}{2} + v}a^{r+v}b^{s+v}}{(q)_r(q)_s(q)_t(q^2;q^2)_v} \\
&= \sum_{r,s,t,v \geq 0} \frac{q^{rt+st+2tv + \binom{t+1}{2} + \binom{r+1}{2} + \binom{s+1}{2} + v}a^{r+v}b^{s+v}d^{r+s+2v}(-q/d)_{r+s+2v}}{(q)_{r}(q)_{s}(q)_t(q^2;q^2)_v} \\
&= \sum_{r,s,t \geq 0} \sum_{v=0}^{\min(r,s)} \frac{q^{rt+st + \binom{t+1}{2} + \binom{r-v+1}{2} + \binom{s-v+1}{2} + v}a^{r}b^{s}d^{r+s}(-q/d)_{r+s}}{(q)_{r-v}(q)_{s-v}(q)_t(q^2;q^2)_v} \\
&= \sum_{r,s,t \geq 0} \sum_{v=0}^{\min(r,s)} \frac{q^{rt+st + \binom{t+1}{2} + \binom{r+1}{2} + \binom{s+1}{2}+ v}a^rb^sd^{r+s}(-q/d)_{r+s}(q^{-r})_v(q^{-s})_v}{(q)_{r}(q)_{s}(q)_t(q^2;q^2)_v} \\
&= \sum_{r,s,t \geq 0} \frac{q^{rt+st+ \binom{t+1}{2} + \binom{r+1}{2} + \binom{s+1}{2}}a^rb^sd^{r+s}(-q/d)_{r+s}(-q)_{r+s}}{(q^2;q^2)_{r}(q^2;q^2)_{s}(q)_t} \\
&= (-q)_{\infty} \sum_{r,s \geq 0 } \frac{q^{r^2+s^2 - \binom{r+s+1}{2}}a^rb^sd^{r+s}(-q/d)_{r+s}}{(q^2;q^2)_r(q^2;q^2)_s},
\end{align*}
which is Theorem \ref{th:capa-asym2}.    

For Theorem \ref{th:capa-asym2nod} we set $d=1$ in the third line of the above string of equations, evaluate the sum over $t$ and then the sums over $u$, $a$, and $b$ as follows:
\begin{align*}
\sum_{r,s,t,v \geq 0} &\frac{q^{rt+st+2tv + \binom{t+1}{2} + \binom{r+1}{2} + \binom{s+1}{2} + v}a^{r+v}b^{s+v}(-q)_{r+s+2v}}{(q)_{r}(q)_{s}(q)_t(q^2;q^2)_v} \\
&= (-q)_{\infty} \sum_{r,s,v \geq 0} \frac{q^{\binom{r+1}{2} + \binom{s+1}{2} + v}a^{r+v}b^{s+v}}{(q)_{r}(q)_{s}(q^2;q^2)_v} \\
&= \frac{(-q)_{\infty}(-aq)_{\infty}(-bq)_{\infty}}{(abq;q^2)_{\infty}}.
\end{align*}

\end{proof}

\section{Some combinatorics and examples}
We begin this section by noting the combinatorial interpretation of the two sums in Theorems \ref{th:capa-chelou} -- \ref{th:capa-asym2}.    First, the sum in \eqref{eq:capa-asym1} is quite straightforward.     The term $a^r/(q^2;q^2)_r$ generates a partition into $r$ non-negative even $a$-parts, and then adding the staircase from $q^{\binom{r+1}{2}}$ gives that $a^rq^{\binom{r+1}{2}}/(q^2;q^2)_r$ generates a partition $\lambda$ into $r$ distinct $a$-parts which alternate in parity.   Here (and throughout this discussion) we assume that there is a ``phantom" $0$ at the end of the partition to account for the parity of the smallest part.   Adding the generalized staircase $d^r(-q/d)_r$ in the usual way then gives an overpartition into distinct parts such that two consecutive parts have different parities if and only if the larger is non-overlined.   The presence of a non-overlined part being equivalent to a parity change, the overlining is in fact unnecessary here.    Let $A(n;k,i)$ denote the number of such partitions of $n$ with $i$ parts and $k$ changes in parity.   Then we have shown that   
\begin{equation} \label{combrsum}
\sum_{n,k,i \geq 0} A(n;k,i)q^na^id^k = \sum_{r \geq 0} \frac{q^{\binom{r+1}{2}}(ad)^r(-q/d)_r}{(q^2;q^2)_r}.
\end{equation}
Note that when $d=0$ there are no changes in parity (or equivalently, only overlined parts), the generating function is $(-aq^2;q^2)_{\infty}$, and Theorem \ref{th:capa-asym1} becomes Theorem \ref{th:capa-aag}.    On the other hand, when there are no restrictions on parity changes (or equivalently, the presence of non-overlined parts), the generating function is $(-aq)_{\infty}$ and we have Theorem \ref{th:capa-asym1nod}.

The double sum in equations \eqref{eq:capa-chelou} and \eqref{eq:capa-asym2} is a little more subtle, but again we will see that the exponent of $d$ counts the number of parity changes.    To begin, the term $a^rb^sq^{r^2+r+s^2+s}/(q^2;q^2)_r(q^2;q^2)_s$ generates a partition $\lambda$ into $r$ distinct even $a$-parts and $s$ distinct even $b$-parts, ordered with the convention that $k_a > k_b$.    The term $q^{-\binom{r+s+1}{2}}$ then means that we remove a staircase of size $r+s$ from $\lambda$.  The resulting partition can be grouped into the same levels as in \eqref{levelorder}, with the extra conditions that there are no $u$-parts and that parts must alternate in parity.    For example, if we begin with  
$$
\lambda = (18_b,16_b,14_a,14_b,12_a,10_a,8_b,4_a,4_b,2_a,2_b),
$$ 
then after removing the staircase we have
$$
(7_b,6_b,5_a ,6_b,5_a,4_a,3_b,0_a,1_b,0_a,1_b).
$$
Now, adding back on the generalized staircase $d^{r+s}(-q/d)_{r+s}$ gives a jagged overpartition counted by $\overline{C}(n;k;i,j)$, where there are no $u$-parts and consecutive parts alternate in parity if and only if the larger is non-overlined.     Let $\overline{C}_4(n;k;i,j)$ denote the number of such jagged overpartitions.   Then we have shown that
\begin{equation} \label{comrssum}
\sum_{n,k,i,j \geq 0} \overline{C}_4(n;k;i,j) a^ib^jd^kq^n = \sum_{r,s \geq 0} \frac{q^{r^2+r+s^2+s - \binom{r+s+1}{2}} a^rb^s d^{r+s} (-q/d)_{r+s}}{(q^2;q^2)_r(q^2;q^2)_s}.
\end{equation}
Obviously when $d=0$ we have the generating function $(-aq^2;q^2)_{\infty}(-bq^2;q^2)_{\infty}$, accounting for Theorem \ref{th:capa-aag}.   On the other hand, when  $d=1$, the number of changes of parity (or equivalently, the number of non-overlined parts) is not restricted anymore.   In this case the pairs $x_a (x-1)_b$ become odd parts $(2x-1)_{ab}$ while any remaining unpaired $y_a$ or $y_b$ are used to build two partitions into distinct parts, giving the generating function  
\begin{equation}
\frac{(-aq)_{\infty}(-bq)_{\infty}}{(abq;q^2)_{\infty}}.
\end{equation}

We turn now to an illustration of Corollaries \ref{cor1} -- \ref{cor3}, beginning with Corollary \ref{cor1}.   Note that the  interpretation of $k$ in the definition of the $\overline{\mathcal{D}}(n;k;i,j)$ comes from the combinatorial interpretation of \eqref{combrsum} above.     First, the thirteen jagged overpartitions of $13$ counted by $\overline{\mathcal{C}}(13;k;i,j)$ are
$$
\begin{gathered}
(\o{13}), (13), (\o{11},\o{2}), (\o{10},\o{3}), (10,\o{3}), (\o{10},1,\o{2}), (10,1,\o{2}), \\
(\o{9},\o{4}), (9,\o{4}), (7,\o{6}), (7,\o{4},\o{2}), (7,4,\o{2}), (4,5,\o{4}),
\end{gathered}
$$
while the thirteen partitions into distinct parts not congruent to $5 \mod 6$ counted by $\overline{\mathcal{D}}(13;k;i,j)$ are 
$$
\begin{gathered}
(13), (12,1), (10,3), (10,2,1), (9,4), (9,3,1), (8,4,1), \\
(8,3,2), (7,6), (7,4,2), (7,3,2,1), (6,4,3), (6,4,2,1).
\end{gathered}
$$
One may then easily verify that for any choice of $k,i,j$,
$$\overline{\mathcal{C}}(13;k;i,j)=\overline{\mathcal{D}}(13;k;i,j).$$
For example, for $k=1$, $i=1$ and $j=0$, we have 
$$\overline{\mathcal{C}}(13;1;1,0)=\overline{\mathcal{D}}(13;1;1,0)=4,$$
the overpartitions counted by $\overline{\mathcal{C}}(13;1;1,0)$ being
$$(13), (10,\o{3}), (9,\o{4}), (7,\o{6})$$
and the partitions counted by $\overline{\mathcal{D}}(13;1;1,0)$ being
$$(13), (12,1), (9,3,1), (7,6).$$

Let us now give an example for Corollary \ref{cor2}.    
First, the twelve jagged overpartitions of $11$ counted by $\overline{\mathcal{C}}'(11;i,j)$ are
$$
\begin{gathered}
(\o{9},2), (9,2), (\o{8},1,2), (8,1,2), (6,5), (6,2,1,2),(5,\o{6}), (5,6),\\
(5,1,2,1,2), (4,5,2), (4,4,1,2), (2,1,2,1,2,1,2),
\end{gathered}
$$
while the twelve overpartitions counted bx $\overline{\mathcal{D}}'(11;i,j)$ are 
$$
\begin{gathered}
(11), (\o{9},\o{2}), (\o{8},3), (8,3), (\o{6},\o{5}), (\o{6},3,\o{2}),\\
(\o{5},3,3), (\o{5},\o{4},\o{2}), (\o{5},4,\o{2}), (4,\o{4},3), (4,4,3), (3,3,3,\o{2}).
\end{gathered}
$$
For any choice of $i,j$, we obtain
$$\overline{\mathcal{C}}'(11;i,j)= \overline{\mathcal{D}}'(11;i,j).$$

We conclude with an example for Corollary \ref{cor3}.    
First, the eight jagged overpartitions of $11$ counted by $\overline{\mathcal{C}}''(11;i,j)$ are
$$
\begin{gathered}
(\o{9},2), (9,2), (6,5), (6,2,1,2),\\
(5,\o{6}), (5,6), (5,1,2,1,2), (2,1,2,1,2,1,2),
\end{gathered}
$$
while the twelve overpartitions counted bx $\overline{\mathcal{D}}''(11;i,j)$ are 
$$
\begin{gathered}
(11), (\o{9},\o{2}), (\o{8},3), (\o{6},\o{5}), (\o{6},3,\o{2}),\\
(\o{5},3,3), (\o{5},\o{4},\o{2}), (3,3,3,\o{2}).
\end{gathered}
$$
Again, for any choice of $i,j$, it is not hard to verify that
$$\overline{\mathcal{C}}''(11;i,j)= \overline{\mathcal{D}}''(11;i,j).$$

\section{Concluding Remarks}
Our use of generalized and partial staircases to generalize the Capparelli identity is motivated by the second author's generalizations of a partition theorem of Schur \cite{Lo2, Lo1}.    In the case of Schur's theorem one obtains ``perfect'' generalizations in the sense that the generating functions for the relevant overpartitions are infinite products without specializing the extra variable $d$.    This is in contrast with our Theorems \ref{th:capa-chelou} -- \ref{th:capa-asym2}, where we have an infinite product for $d=0$ and $d=1$, but not for all $d$.    Perhaps the most natural question arising from this work is whether generalized and/or partial staircases yield interesting results in the context of other partition identities which have been studied using the method of weighted words, such as G\"ollnitz' (big) theorem \cite{Al-An-Go1},  Siladi\'c's identity \cite{Do1}, Primc's identity \cite{Do-Lo1}, or the Alladi-Andrews-Berkovich identity \cite{Al-An-Be1}.    The only result we are aware of in this direction is an ``imperfect'' generalization of G\"ollnitz' theorem in \cite{Lo2} using a generalized staircase.


\begin{thebibliography}{99}

\bibitem{Al-An-Be1}
K. Alladi, G. E. Andrews, and A. Berkovich, A new four parameter $q$-series identity and its partition implications, \emph{Invent. Math.} {\bf 153} (2003), no. 2, 231--260.

\bibitem{Al-An-Go.5}
K. Alladi, G. E. Andrews, and B. Gordon, Generalizations and refinements of a partition theorem of G\"ollnitz, \emph{J. Reine Angew. Math.} {\bf 460} (1995), 165--188.

\bibitem{Al-An-Go1}
K. Alladi, G. E. Andrews, and B. Gordon, Refinements and generalizations of Capparelli's conjecture on partitions, \emph{J. Algebra} {\bf 174} (1995), 636--658.

\bibitem{Andrews1}
G. E. Andrews, {S}chur’s theorem, {C}apparelli’s conjecture and $q$-trinomial coefficients, \textit{Contemp. Math.}, \textbf{166} (1992), 141--154.

\bibitem{Be-Un1}
A. Berkovich and A. K. Uncu, A new companion to Capparelli's identities, \emph{Adv. in Appl. Math.} {\bf 71} (2015), 125--137.

\bibitem{Br-Ma1}
K. Bringmann and K. Mahlburg,  False theta functions and companions to Capparelli's identities, \emph{Adv. Math.} {\bf 278} (2015), 121--136.

\bibitem{Capparelli1}
S. Capparelli, On some representations of twisted affine Lie algebras and combinatorial identities, \textit{J. Algebra}, \textbf{154} (1993), 335--355.

\bibitem{Capparelli2}
S. Capparelli, A construction of the level $3$ modules for the affine Lie algebra $A_2^{(2)}$ and a new combinatorial identity of the Rogers-Ramanujan type, \textit{Trans. Amer. Math. Soc.}, \textbf{348} (1996), 481--501.

\bibitem{Do1}
J. Dousse, Siladi\'c's theorem: weighted words, refinement and companion , \emph{Proc. Amer. Math. Soc.}, to appear.

\bibitem{Do-Lo1}
J. Dousse and J. Lovejoy, On a Rogers-Ramanujan type identity from crystal base theory, \emph{Proc. Amer. Math. Soc.}, to appear.


\bibitem{Fo-Ja-Ma1}
J.-F. Fortin, P. Jacob, and P. Mathieu, Jagged partitions, \emph{Ramanujan J.} {\bf 10} (2005), no. 2, 215--235.

\bibitem{Fu-Ze1}
S. Fu and J. Zeng, A unifying combinatorial approach to refined little Göllnitz and Capparelli's companion identities, preprint, \url{arXiv:1603.07068v2}. 

\bibitem{Ga-Ra1}
G. Gasper and M. Rahman, Basic Hypergeometric Series, Second edition, Encyclopedia of Mathematics and its Applications 96,  Cambridge University Press, Cambridge, 2004.

\bibitem{Lepowsky}
J. Lepowsky and R. Wilson, The structure of standard modules, {I}:
  Universal algebras and the {R}ogers-{R}amanujan identities, \emph{Invent. Math.}, \textbf{77} (1984), 199--290.

\bibitem{Lo.5}
J. Lovejoy, Constant terms, jagged partitions, and partitions with difference two at distance two, \emph{Aequationes Math.} {\bf 72} (2006), no. 3, 299--312.

\bibitem{Lo2}
J. Lovejoy, A theorem on seven-colored overpartitions and its applications, \textit{Int. J. Number Theory} \textbf{1} (2005), 215--224.

\bibitem{Lo1}
J. Lovejoy, Asymmetric generalizations of Schur's theorem, \emph{Proceedings of the Alladi60 conference}, to appear.


\bibitem{Meurman}
A. Meurman and M. Primc, Annihilating ideals of standard modules of {$sl(2,{\mathbb C})^\sim$} and combinatorial identities, \textit{Adv.  Math.}, \textbf{64} (1987), 177--240.

\bibitem{Primc}
M. Primc, Some crystal {R}ogers-{R}amanujan type identities, \textit{Glas. Math. Ser. III}, \textbf{34} (1999), 73--86.

\bibitem{RogersRamanujan}
L. J. Rogers and S. Ramanujan, Proof of certain identities in combinatory analysis, \emph{Cambr. Phil. Soc. Proc.}, \textbf{19} (1919), 211--216.


\bibitem{Siladic}
I. Siladi\'c, Twisted {$sl(3,\mathbb{C})^{\sim}$}-modules and combinatorial identities, \url{arXiv:math/0204042v2}.

\bibitem{Si1}
A. V. Sills, On series expansions of Capparelli's infinite product, \emph{Adv. in Appl. Math.} {\bf 33} (2004), no. 2, 397--408.

\bibitem{Xie}
M. Tamba and C. Xie, Level three standard modules for $A_2^{(2)}$ and combinatorial identities, \textit{J. Pure Applied Algebra}, \textbf{105} (1995), 53--92.


\end{thebibliography}
\end{document}